\documentclass[a4paper,12pt]{amsart}
\usepackage{amsmath,amssymb,amsthm,verbatim}
\usepackage[mathscr]{eucal}
\usepackage{mathpazo}
\usepackage[T1]{fontenc}

\theoremstyle{plain}
\newtheorem{theorem}{Theorem}
\newtheorem{lemma}[theorem]{Lemma}
\newtheorem{proposition}[theorem]{Proposition}
\newtheorem{corollary}[theorem]{Corollary}
\theoremstyle{definition}
\newtheorem*{remark}{Remark}

\newcommand{\M}{\mathscr{M}}

\newcommand{\E}{\mathfrak{E}}
\newcommand{\BH}{\mathbb{B}(\mathcal{H})}
\newcommand{\Ha}{\mathcal{H}}
\newcommand{\al}{\alpha}
\newcommand{\1}{\mathbb{1}}
\newcommand{\ro}{\rho}

\newcommand{\tr}{\operatorname{tr}}
\renewcommand{\geq}{\geqslant}
\renewcommand{\leq}{\leqslant}
\DeclareMathOperator{\Lin}{Lin}

\DeclareMathOperator{\s}{s}

\begin{document}
\title[On Segal entropy]{On Segal entropy}
\author{Andrzej \L uczak}
\address{Faculty of Mathematics and Computer Science\\
         \L\'od\'z University\\
         ul. S. Banacha 22\\
         90-238 \L\'od\'z, Poland}

\email[Andrzej \L uczak]{andrzej.luczak@wmii.uni.lodz.pl}
\thanks{}
\date{}
\begin{abstract}
 The paper is devoted to the investigation of Segal's entropy in semifinite von Neumann algebras. The following questions are dealt with: semicontinuity, the 'ideal-like' structure of  the linear span of the set of operators with finite entropy, and topological properties of the set of operators with finite as well as infinite entropy. In our analysis, full generality is aimed at, in particular, the operators for which the entropy is considered are not assumed to belong to the underlying von Neumann algebra, instead, they are arbitrary positive elements  in the space $L^1$ over the algebra.
\end{abstract}
\maketitle

\section*{Introduction}
In 1960 I. Segal \cite{S} introduced the notion of entropy for semifinite von Neumann algebras being a straightforward generalisation of the von Neumann entropy defined for the full algebra $\BH$ of all bounded linear operators on a Hilbert space by means of the canonical trace. While von Neumann's entropy became an all-important object for quantum mechanics, the role of Segal's entropy has been less significant, presumably because of the more general setup of an arbitrary semifinite von Neumann algebras in which this entropy appears. However, the rapid development of operator-algebraic methods in quantum theory where arbitrary von Neumann algebras play a crucial role makes it highly probable that also the Segal entropy will gain much more importance and interest.

Despite the obvious similarity in the definitions of both entropies, in the case of an arbitrary semifinite von Neumann algebra, where instead of the canonical trace we have a normal semifinite faithful trace, substantial differences between them arise. Perhaps the most fundamental one consists in the fact that while a normal (non-normalised) state on $\BH$ is represented by a positive operator of trace-class (the so-called `density matrix'), in the case of an arbitrary semifinite von Neumann algebra this `density matrix' can be an unbounded operator. This prompted Segal to consider only the states whose `density matrices' were in the algebra. In our analysis, we avoid this restriction.

In the paper, the following questions concerning Segal's entropy are investigated: its semicontinuity, the 'ideal-like' structure of the linear span of the set of operators with finite entropy, and topological properties of the set of operators with finite as well as infinite entropy. In particular, we show that this entropy is a function of first Baire's class, and give sufficient conditions for its lower and upper semicontinuity. Since von Neumann's entropy is a particular instance of Segal's entropy, our results yield also a condition for the continuity of von Neumann's entropy. Further, we show that for a finite von Neumann algebra, the linear span of the set of operators with finite entropy is invariant with respect to multiplication by the operators from the algebra. Finally, it is shown that the sets of elements with finite as well as infinite entropy are dense in the space of positive integrable operators, and that in the case when the underlying algebra is finite, the set of elements with finite entropy is of the first category --- both these properties are known to hold for von Neumann entropy.

\section{Preliminaries and notation}
Let $\M$ be a semifinite von Neumann algebra of operators acting on a Hilbert space $\Ha$, with a normal semifinite faithful trace $\tau$, identity $\1$, and predual $\M_*$. By $\|\cdot\|_\infty$ we shall denote the operator norm on $\M$. The set of positive functionals in $\M_*$ shall be denoted by $\M_*^+$. These functionals are sometimes referred to as (non-normalised) sta\-tes.

A densely defined closed operator $x$ is said to be \emph{affiliated with} $\M$ if for the polar decomposition
\[
 x=u|x|
\]
we have $u\in\M$, and the spectral projections of $|x|$ belong to $\M$. An operator $x$ affiliated with $\M$ is said to be \emph{measurable} if for some $\lambda_0$ we have $\tau(e[\lambda_0,\infty))<+\infty$, where $e([\lambda_0,\infty))$ is the spectral projection of $|x|$ corresponding to the interval $[\lambda_0,\infty)$.

The algebra of measurable operators $\widetilde{\M}$ is defined as a topological ${}^*$-algebra of operators on $\mathcal{H}$ affiliated with $\M$ with strong addition $\dotplus$ and strong multiplication $\cdot$, i.e.
\[
 x\dotplus y=\overline{x+y},\qquad x\cdot y=\overline{xy},\qquad x,y\in\widetilde{\M},
\]
where $\overline{x+y}$ and $\overline{xy}$ are the closures of the corresponding operators defined by addition and composition on the natural domains given by the intersections of the domains of the $x$ and $y$ and of the range of $y$ and the domain of $x$, respectively. It is known that such closures exist and give operators affiliated with $\M$. In what follows, we shall omit the dot in the symbols of these operations and write simply $x+y$ and $xy$ to denote $x\dotplus y$ and $x\cdot y$. In particular, for a finite von Neumann algebra $\M$ all operators affiliated with $\M$ are measurable. A great advantage, while dealing with measurable operators, is usually a lack of problems concerning their domains since for any such operators there is a common domain which is a core for them.

The domain of a linear operator $x$ on $\Ha$ will be denoted by $\mathcal{D}(x)$.

For each $\ro\in\M_*$, there is an operator $h$ affiliated with $\M$ such that
\[
 \ro(x)=\tau(xh)=\tau(hx), \quad x\in\M.
\]
The space of all such operators is denoted by $L^1(\M,\tau)$, and the correspondence above is one-to-one and isometric, where the norm on $L^1(\M,\tau)$, denoted by $\|\cdot\|_1$, is defined as
\[
 \|h\|_1=\tau(|h|), \quad h\in L^1(\M,\tau).
\]
(In the theory of noncommutative $L^p$-spaces for semifinite von Neumann algebras, it it shown that $\tau$ can be extended to the $h$'s as above; see  e.g. \cite{N,T,Y} for a detailed account of this theory.) Moreover, to hermitian functionals in $\M_*$ correspond selfadjoint operators in $L^1(\M,\tau)$, and to states in $\M_*$ --- positive operators in $L^1(\M,\tau)$. The set of these operators will be denoted by $L^1(\M,\tau)^+$.

For $\ro\in\M_*^+$, the corresponding element in $L^1(\M,\tau)^+$, called the \emph{density} of $\ro$, is denoted by $h_\ro$.

The \emph{Segal entropy} of $\ro$, denoted by $H(\ro)$, is defined as
\[
 H(\ro)=\tau(h_\ro\log h_\ro),
\]
i.e. for the spectral representation of $h_\ro$
\begin{equation}\label{spec}
 h_\ro=\int_0^\infty t\,e(dt),
\end{equation}
we have
\[
 H(\ro)=\int_0^\infty t\log t\,\tau(e(dt)).
\]
Accordingly, we define Segal's entropy for $h\in L^1(\M,\tau)^+$ by the formula
\[
 H(h)=\tau(h\log h)=\int_0^\infty t\log t\,\tau(e(dt)),
\]
where $h$ has the spectral representation as in \eqref{spec}. Obviously, Segal's entropy need not be defined for all $h\in L^1(\M,\tau)^+$. However, for finite von Neumann algebras we have, on account of the inequality
\[
 t\log t\geq t-1,
\]
the relation
\begin{align*}
 H(h)&=\int_0^\infty t\log t\,\tau(e(dt))\geq\int_0^\infty(t-1)\,\tau(e(dt))\\
 &=\tau\Big(\int_0^\infty t\,e(dt)\Big)-\tau\Big(\int_0^\infty e(dt)\Big)=\tau(h)-1>-\infty,
\end{align*}
showing that in this case Segal's entropy is well-defined (and nonnegative for the normalised states). It should be noted that the original Segal definition of entropy differs from ours by a minus sign before the trace. However, for the sake of having nonnegative entropy for normalised states on a finite von Neumann algebra we have adopted the definition as above. Let us recall that for $\M=\BH$ and the canonical trace `$\tr$' the von Neumann entropy of the density matrix $h$ with spectral decomposition
\[
 h=\sum_{n=1}^\infty \lambda_ne_n
\]
is defined as
\[
 S(h)=-\tr h\log h=-\sum_{n=1}^\infty \lambda_n\log\lambda_n\tr e_n,
\]
so for $\tau=\tr$ we have
\[
 H(h)=-S(h).
\]
Observe that von Neumann's entropy is well-defined on the whole of $L^1(\M,\tau)^+=$ nonnegative trace-class operators, since each such operator has its eigenvalues $\lambda_n$ converging to zero, so there are only a finite number of positive elements of the form $\lambda_n\log\lambda_n$.

In what follows, we shall speak of Segal's entropy for elements in $L^1(\M,\tau)^+$ instead of states.

\section{Semicontinuity of Segal entropy}
Put, for simplicity of notation,
\[
 f(t)=t\log t, \quad t\in[0,\infty).
\]
Then
\[
 H(h)=\tau(h\log h)=\tau(f(h))
\]
whenever $\tau(f(h))$ is well-defined. By $\E$ we shall denote  the set of elements in $L^1(\M,\tau)^+$ for which the entropy $H(h)$ exists while $\E_{\text{fin}}$ will stand for the set of elements with finite entropy, and $\E_\infty$ --- for the set of elements with infinite entropy; thus
\[
 \E=\E_{\text{fin}}\cup\E_\infty.
\]
In particular, for $\M$ finite or $\M=\BH$ we have
\[
 \E=L^1(\M,\tau)^+.
\]

For arbitrary $0<m<M<+\infty$, let $f_{m,M}\colon L^1(\M,\tau)^+\to L^1(\M,\tau)$ be defined as
\begin{align*}
 f_{m,M}(h)&=h\log(m\1+h)-h\log(m+1)\\
 &+h\log(M+1)-h\log(M\1+h).
\end{align*}
In \cite{LP2}, it was shown that for this function the following representation holds
\begin{equation}\label{fmM}
 f_{m,M}(h)=\int_m^M\Big(\frac{1}{s+1}h-h(s\1+h)^{-1}\Big)\,ds,
\end{equation}
where the integral is Bochner's integral of a function with values in the Banach space $L^1(\M,\tau)$. (This representation follows from the formula
\begin{align*}
 \int_m^M\Big(\frac{t}{s+1}-\frac{t}{s+t}\Big)\,ds&=t\log(m+t)-t\log(m+1)\\
 &+t\log(M+1)-t\log(M+t),
\end{align*}
and the estimate
\begin{align*}
 \Big\|\frac{1}{s+1}h-h(s\1+h)^{-1}\Big\|_1&=\frac{1}{s+1}\|h(h-\1)(s\1+h)^{-1}\|_1\\
 &\leq\frac{1}{s+1}\|h\|_1\|(h-1)(s\1+h)^{-1}\|_\infty\\
 &=\frac{1}{s+1}\max\Big\{\frac{1}{s},1\Big\}\|h\|_1,
\end{align*}
yielding Bochner's integrability of the function under the integral sign in the formula \eqref{fmM}.) The idea behind the function $f_{m,M}$ is that $f_{m,M}(h)$ should approximate $f(h)$ as $m\to0$ and $M\to\infty$. It was shown in \cite[Lemma 6]{LP2} that indeed if $h\log h\in L^1(\M,\tau)$, then
\[
 \lim_{\substack{m\to0\\M\to\infty}}f_{m,M}(h)=f(h) \quad \text{in }\|\cdot\|_1\text{-norm}.
\]
We shall show that a similar result is possible for the entropy $H(h)$ whenever it is well-defined. A slightly simplified version of the approximation of $H(h)$ by $\tau(f_{m,M}(h))$ is as follows.
\begin{lemma}\label{L}
Let $h\in\E$. Then
\[
 \lim_{M\to\infty}\tau(f_{1/M,M}(h))=\tau(f(h)).
\]
\end{lemma}
\begin{proof}
Certainly, we may assume that $\tau(f(h))=\pm\infty$. Let
\[
 h=\int_0^\infty t\,e(dt)
\]
be the spectral representation of $h$. Then
\begin{align*}
 f_{m,M}(h)&=h\log(m\1+h)-h\log(m+1)+h\log(M+1)+\\
 &-h\log(M\1+h)=\int_0^\infty t\log\frac{(M+1)(m+t)}{(m+1)(M+t)}\,e(dt),
\end{align*}
and for $m=\frac{1}{M}$ we have
\[
 f_{m,M}(h)=\int_0^\infty t\log\frac{Mt+1}{M+t}\,e(dt).
\]
Furthermore,
\begin{align*}
 \tau(h\log h)&=\int_0^\infty t\log t\,\tau(e(dt))\\
 &=\int_0^1t\log t\,\tau(e(dt))+\int_1^\infty t\log t\,\tau(e(dt)).
\end{align*}
Assume first that $\tau(f(h))=+\infty$. Then
\[
 -\infty<\int_0^1t\log t\,\tau(e(dt))\leq0 \quad \text{and} \quad \int_1^\infty t\log t\,\tau(e(dt))=+\infty.
\]
We have
\begin{align*}
 \tau(f_{m,M}(h))&=\int_0^\infty t\log\frac{Mt+1}{M+t}\,\tau(e(dt))\\
 &=\int_0^1t\log\frac{Mt+1}{M+t}\,\tau(e(dt))+\int_1^\infty t\log\frac{Mt+1}{M+t}\,\tau(e(dt)).
\end{align*}
For $0\leq t\leq1$, we have
\[
 t\leq\frac{Mt+1}{M+t}\leq1,
\]
thus
\[
 0\geq\int_0^1t\log\frac{Mt+1}{M+t}\,\tau(e(dt))\geq\int_0^1t\log t\,\tau(e(dt))>-\infty.
\]
Hence
\[
 \tau(f_{m,M}(h))\geq\int_0^1t\log t\,\tau(e(dt))+\int_1^\infty t\log\frac{Mt+1}{M+t}\,\tau(e(dt)).
\]
For the second integral we have, using Fatou's Lemma,
\begin{align*}
 \liminf_{M\to\infty}\int_1^\infty t\log\frac{Mt+1}{M+t}\,\tau(e(dt))&\geq\int_1^\infty\liminf_{M\to\infty}t\log\frac{Mt+1}{M+t}\,\tau(e(dt))\\
 &=\int_1^\infty t\log t\,\tau(e(dt))=+\infty,
\end{align*}
showing that
\[
 \lim_{M\to\infty}\int_1^\infty t\log\frac{Mt+1}{M+t}\,\tau(e(dt))=+\infty.
\]
Consequently,
\begin{align*}
 \liminf_{M\to\infty}\tau(f_{m,M}(h))&\geq\int_0^1t\log t\,\tau(e(dt))\\
 &+\lim_{M\to\infty}\int_1^\infty t\log\frac{Mt+1}{M+t}\,\tau(e(dt))=+\infty,
\end{align*}
i.e.
\[
 \lim_{M\to\infty}\tau(f_{m,M}(h))=+\infty=\tau(f(h)).
\]

Now let $\tau(f(h))=-\infty$.Then
\[
 \int_0^1t\log t\,\tau(e(dt))=-\infty \quad \text{and} \quad 0\leq\int_1^\infty t\log t\,\tau(e(dt))<+\infty.
\]
For $t\geq1$, we have
\[
 1\leq\frac{Mt+1}{M+t}\leq t,
\]
thus
\[
 0\leq\int_1^\infty t\log\frac{Mt+1}{M+t}\,\tau(e(dt))\leq\int_1^\infty t\log t\,\tau(e(dt))<+\infty.
\]
Hence
\begin{align*}
 \tau(f_{m,M}(h))&=\int_0^1t\log\frac{Mt+1}{M+t}\,\tau(e(dt))+\int_1^\infty t\log\frac{Mt+1}{M+t}\,\tau(e(dt))\\
 &\leq\int_0^1t\log\frac{Mt+1}{M+t}\,\tau(e(dt))+\int_1^\infty t\log t\,\tau(e(dt)).
\end{align*}
For the first integral in the formula above we have, again using Fatou's Lemma,
\begin{align*}
 -&\limsup_{M\to\infty}\int_0^1t\log\frac{Mt+1}{M+t}\,\tau(e(dt))\\
 =&\liminf_{M\to\infty}\int_0^1\Big(-t\log\frac{Mt+1}{M+t}\Big)\,\tau(e(dt))\\
 \geq&\int_0^1\liminf_{M\to\infty}\Big(-t\log\frac{Mt+1}{M+t}\Big)\,\tau(e(dt))\\
 =&\int_0^1(-t\log t)\,\tau(e(dt))=+\infty,
\end{align*}
showing that
\[
 \lim_{M\to\infty}\int_0^1t\log\frac{Mt+1}{M+t}\,\tau(e(dt))=-\infty.
\]
Consequently,
\begin{align*}
 \limsup_{M\to\infty}\tau(f_{m,M}(h))&\leq\int_0^1t\log t\,\tau(e(dt))\\
 &+\lim_{M\to\infty}\int_1^\infty t\log\frac{Mt+1}{M+t}\,\tau(e(dt))=-\infty,
\end{align*}
i.e.
\[
 \lim_{M\to\infty}\tau(f_{m,M}(h))=-\infty=\tau(f(h)). \qedhere
\]
\end{proof}
Our next result concerns the continuity of the function $f_{m,M}$.
\begin{proposition}\label{P}
The function $f_{m,M}$ is uniformly continuous in $\|\cdot\|_1$-norm on $L^1(\M,\tau)^+$.
\end{proposition}
\begin{proof}
Let $h',h''\in L^1(\M,\tau)^+$ be arbitrary. We have
\begin{align*}
 &(s\1+h')^{-1}-(s\1+h'')^{-1}\\
 =&(s\1+h')^{-1}((s\1+h'')-(s\1+h'))(s\1+h'')^{-1}\\
 =&(s\1+h')^{-1}(h''-h')(s\1+h'')^{-1},
\end{align*}
so we obtain
\begin{align*}
 &f_{m,M}(h')-f_{m,M}(h'')=\int_m^M\Big(\frac{1}{s+1}h'-h'(s\1+h')^{-1}\Big)\,ds\\
  &\phantom{f_{m,M}(h')-f_{m,M}(h'')}-\int_m^M\Big(\frac{1}{s+1}h''-h''(s\1+h'')^{-1}\Big)\,ds\\
  &\phantom{f_{m,M}(h')-f_{m,M}(h'')}=\int_m^M\Big(\frac{1}{s+1}h'-\1+s(s\1+h')^{-1}\Big)\,ds\\
  &\phantom{f_{m,M}(h')-f_{m,M}(h'')}-\int_m^M\Big(\frac{1}{s+1}h''-\1+s(s\1+h'')^{-1}\Big)\,ds\\
 =&\int_m^M\Big(\frac{1}{s+1}(h'-h'')+s\big((s\1+h')^{-1}-(s\1+h'')^{-1}\big)\Big)\,ds\\
 =&\int_m^M\Big(\frac{1}{s+1}(h'-h'')+s\big((s\1+h')^{-1}(h''-h')(s\1+h'')^{-1}\big)\Big)\,ds.
\end{align*}
The inequality
\begin{align*}
 &\big\|(s\1+h')^{-1}(h''-h')(s\1+h'')^{-1}\big\|_1\\
 \leq&\big\|(s\1+h')^{-1}\|_\infty\|h''-h'\|_1\|(s\1+h'')^{-1}\big\|_\infty\leq\frac{1}{s^2}\|h'-h''\|_1
\end{align*}
yields
\begin{align*}
 &\|f_{m,M}(h')-f_{m,M}(h'')\|_1\\
 \leq&\int_m^M\Big\|\frac{1}{s+1}(h'-h'')+s\big((s\1+h')^{-1}(h''-h')(s\1+h'')^{-1}\big)\Big\|_1\,ds\\
 \leq&\int_m^M\Big(\frac{1}{s+1}\|h'-h''\|_1+s\big\|(s\1+h')^{-1}(h''-h')(s\1+h'')^{-1}\big\|_1\Big)\,ds\\
 \leq&\int_m^M\Big(\frac{1}{s+1}+\frac{1}{s}\Big)\|h'-h''\|_1\,ds=\|h'-h''\|_1\int_m^M\Big(\frac{1}{s+1}+\frac{1}{s}\Big)\,ds,
\end{align*}
and the conclusion follows.
\end{proof}
As a corollary to Lemma~\ref{L} and Proposition~\ref{P} we obtain
\begin{theorem}
The function $H$ is of first Baire's class on $\E$.
\end{theorem}
\begin{proof}
The proof follows from the relation
\[
 H(h)=\lim_{M\to\infty}\tau(f_{1/M,M}(h)),
\]
and the continuity of the functions $f_{1/M,M}$, $M=2,3,\dots$, in \linebreak $\|\cdot\|_1$-norm yielding the continuity of the functions $\tau\circ f_{1/M,M}$.
\end{proof}
From the properties of the functions of Baire's first class and the fact that for $\BH$ the set $\E=$ nonnegative trace-class operators is complete, we get
\begin{corollary}
The set of (generalised) continuity points of von Neumann's entropy is dense in the set of nonnegative trace-class operators on $\Ha$.
\end{corollary}
(By a generalised continuity point of an extended real-valued function g is meant a point $t_0$ for which the equality $\displaystyle{\lim_{t\to t_0}g(t)=g(t_0)}$ holds, where we admit the possibilities $g(t)=\pm\infty$ and $g(t_0)=\pm\infty$.) The same result holds also for Segal's entropy and $\M$ finite, however in this case we have more as follows from point (3) of the next theorem which gives some conditions for semicontinuity of Segal's entropy. In particular, point (1) of the theorem generalises considerably Theorem 9 in \cite{LP1}.
\begin{theorem}\label{u-lscont}
Let $\varepsilon>0$ and $c>0$ be arbitrary.
\begin{enumerate}
 \item Segal's entropy is upper-semicontinuous on the set\\ $\{h\in\E:\tau(h^{1+\varepsilon})\leq c\}$.
 \item Segal's entropy is lower-semicontinuous on the set\\ $\{h\in\E:\tau(h^{1-\varepsilon})\leq c\}$.
 \item For $\M$ finite, Segal's entropy is lower-semicontinuous on the whole of $L^1(\M,\tau)^+$.
\end{enumerate}
\end{theorem}
\begin{proof}
(1) Let $h_0\in\{h\in\E:\tau(h^{1+\varepsilon})\leq c\}$, and let $(h_n)$ be an arbitrary sequence in $\{h\in\E:\tau(h^{1+\varepsilon})\leq c\}$ converging to $h_0$ in $\|\cdot\|_1$-norm. We have
\begin{equation}\label{fhn}
 f(h_n)=f(h_n)-f_{m,M}(h_n)+f_{m,M}(h_n)-f_{m,M}(h_0)+f_{m,M}(h_0).
\end{equation}
Furthermore the following estimate holds
\begin{align*}
 &f(h_n)-f_{m,M}(h_n)=h_n\log h_n-h_n\log(m\1+h_n)\\
 +&h_n\log(m+1)-h_n\log(M+1)+h_n\log(M\1+h_n)\\
 \leq&h_n\log(m+1)-h_n\log(M+1)+h_n\log(M\1+h_n),
\end{align*}
since
\[
 h_n\log h_n-h_n\log(m\1+h_n)\leq0.
\]
Let
\[
 h_n=\int_0^\infty t\,e_n(dt)
\]
be the spectral representation of $h_n$. We have
\begin{align*}
 &\tau(h_n\log(M\1+h_n)-h_n\log(M+1))=\int_0^\infty t\log\frac{M+t}{M+1}\,\tau(e_n(dt))\\
 =&\int_0^1t\log\frac{M+t}{M+1}\,\tau(e_n(dt))+\int_1^\infty t\log\frac{M+t}{M+1}\,\tau(e_n(dt)\\
 \leq&\int_1^\infty t\log\frac{M+t}{M+1}\,\tau(e_n(dt))\leq\int_1^\infty t\log\Big(1+\frac{t}{M}\Big)\,\tau(e_n(dt)).
\end{align*}
For $\varepsilon>0$ there is $r=r(\varepsilon)$ such that for $u>r(\varepsilon)$ we have
\[
 \log(1+u)\leq u^\varepsilon.
\]
Hence, for $\frac{t}{M}>r(\varepsilon)$, i.e. $t>Mr(\varepsilon)$, we have the estimate
\[
 t\log\Big(1+\frac{t}{M}\Big)\leq\frac{t^{1+\varepsilon}}{M^\varepsilon}.
\]
For $t\leq Mr(\varepsilon)$, we have
\begin{align*}
 t\log\Big(1+\frac{t}{M}\Big)\leq\frac{t^2}{M}=\frac{t^{1-\varepsilon}\,t^{1+\varepsilon}}{M}
 \leq\frac{[Mr(\varepsilon)]^{1-\varepsilon}\,t^{1+\varepsilon}}{M}
 =\frac{r(\varepsilon)^{1-\varepsilon}\,t^{1+\varepsilon}}{M^\varepsilon}.
\end{align*}
Now
\begin{align*}
 \int_1^\infty t\log\Big(1+\frac{t}{M}\Big)\,\tau(e_n(dt))&=\int_1^{Mr(\varepsilon)}t\log\big(1+\frac{t}{M}\Big)\,\tau(e_n(dt))\\
 &+\int_{Mr(\varepsilon)}^\infty t\log\Big(1+\frac{t}{M}\Big)\,\tau(e_n(dt)),
\end{align*}
and the two integrals above are estimated as follows
\begin{align*}
 \int_1^{Mr(\varepsilon)}t\log\big(1+\frac{t}{M}\Big)\,\tau(e_n(dt))&\leq\frac{r(\varepsilon)^{1-\varepsilon}}{M^\varepsilon}\int_1^{Mr(\varepsilon)}t^{1+\varepsilon}\,\tau(e_n(dt))\\
 &\leq\frac{r(\varepsilon)^{1-\varepsilon}}{M^\varepsilon}\tau\big(h_n^{1+\varepsilon}\big),
\end{align*}
and
\begin{align*}
 \int_{Mr(\varepsilon)}^\infty t\log\Big(1+\frac{t}{M}\Big)\,\tau(e_n(dt))&\leq\frac{1}{M^\varepsilon}\int_{Mr(\varepsilon)}^\infty t^{1+\varepsilon}\,\tau(e_n(dt))\\
 &\leq\frac{1}{M^\varepsilon}\tau\big(h_n^{1+\varepsilon}\big).
\end{align*}
Consequently,
\begin{align*}
 \int_1^\infty t\log\Big(1+\frac{t}{M}\Big)\,\tau(e_n(dt))&\leq\frac{r(\varepsilon)^{1-\varepsilon}}{M^\varepsilon}\tau\big(h_n^{1+\varepsilon}\big)+\frac{1}{M^\varepsilon}
 \tau\big(h_n^{1+\varepsilon}\big)\\
 &=\frac{r(\varepsilon)^{1-\varepsilon}+1}{M^\varepsilon}\tau\big(h_n^{1+\varepsilon}\big),
\end{align*}
which gives the following estimate for the first term in the equality~\eqref{fhn}
\[
 \tau(f(h_n)-f_{m,M}(h_n))\leq\log(m+1)\tau(h_n)+\frac{r(\varepsilon)^{1-\varepsilon}+1}{M^\varepsilon}\tau\big(h_n^{1+\varepsilon}\big).
\]
As for the second term, Proposition \ref{P} yields
\[
 \tau(f_{m,M}(h_n)-f_{m,M}(h_0))\to0 \quad \text{as }n\to\infty.
\]
Taking into account the estimates above, we get
\begin{align*}
 &\tau(f(h_n))=\tau(f(h_n)-f_{m,M}(h_n))+\tau(f_{m,M}(h_n)-f_{m,M}(h_0))\\
 +&\tau(f_{m,M}(h_0))\leq\log(m+1)\tau(h_n)+\frac{r(\varepsilon)^{1-\varepsilon}+1}{M^\varepsilon}\tau\big(h_n^{1+\varepsilon}\big)\\
 +&\tau(f_{m,M}(h_n)-f_{m,M}(h_0))+\tau(f_{m,M}(h_0))\leq\log(m+1)\tau(h_n)\\
 +&c\frac{r(\varepsilon)^{1-\varepsilon}+1}{M^\varepsilon}+\tau(f_{m,M}(h_n)-f_{m,M}(h_0))+\tau(f_{m,M}(h_0)).
\end{align*}
Now passing to the limit with $n\to\infty$ in the inequality above yields
\begin{equation}\label{lsup}
 \begin{aligned}
  \limsup_{n\to\infty}\tau(f(h_n))&\leq\log(m+1)\tau(h_0)\\
  &+c\frac{r(\varepsilon)^{1-\varepsilon}+1}{M^\varepsilon}+\tau(f_{m,M}(h_0)).
 \end{aligned}
\end{equation}
Taking $m=\frac{1}{M}$ and again passing to the limit with $M\to\infty$ in the inequality above, we get on account of Lemma \ref{L}
\begin{align*}
 \limsup_{n\to\infty}H(h_n)&=\limsup_{n\to\infty}\tau(f(h_n))\\
 &\leq\lim_{M\to\infty}\tau(f_{m,M}(h_0))=\tau(f(h_0))=H(h_0),
\end{align*}
proving the upper-semicontinuity of $H$.

(2) The overall strategy of the proof is the same as in part (1), the only difference lies in estimates. Let $h_0\in\{h\in\E:\tau(h^{1-\varepsilon})\leq c\}$, and let $(h_n)$ be an arbitrary sequence in the set $\{h\in\E:\tau(h^{1-\varepsilon})\leq c\}$ converging to $h_0$ in $\|\cdot\|_1$-norm. For $h\in L^1(\M,\tau)^+$ with spectral representation
\[
 h=\int_0^\infty t\,e(dt),
\]
we have for each $a>0$
\begin{align*}
 \tau(h^{1-\varepsilon})&=\int_0^\infty t^{1-\varepsilon}\,\tau(e(dt))
 =\int_0^at^{1-\varepsilon}\,\tau(e(dt))+\int_a^\infty t^{1-\varepsilon}\,\tau(e(dt))\\
 &\geq\int_0^at^{1-\varepsilon}\,\tau(e(dt))+a^{1-\varepsilon}\tau(e([a,\infty))),
\end{align*}
which means that
\begin{equation}\label{c1}
 a^{1-\varepsilon}\tau(e([a,\infty)))\leq\tau(h^{1-\varepsilon})
\end{equation}
for every $a>0$. We have
\[
 h_n\log(m\1+h_n)-h_n\log h_n=\int_0^\infty t\log\Big(1+\frac{m}{t}\Big)\,e_n(dt).
\]
For $\frac{m}{t}\geq r(\varepsilon)$, i.e. $t\leq\frac{m}{r(\varepsilon)}$, we have
\[
 \log\Big(1+\frac{m}{t}\Big)\leq\frac{m^\varepsilon}{t^\varepsilon},
\]
thus
\begin{align*}
 &\int_0^{m/r(\varepsilon)}t\log\Big(1+\frac{m}{t}\Big)\,\tau(e_n(dt))\leq\int_0^{m/r(\varepsilon)}t\,\frac{m^\varepsilon}{t^\varepsilon}\,\tau(e_n(dt))\\
 =&m^\varepsilon\int_0^{m/r(\varepsilon)}t^{1-\varepsilon}\tau(e_n(dt))\leq m^\varepsilon\tau(h_n^{1-\varepsilon})\leq cm^\varepsilon.
\end{align*}
For $t\geq\frac{m}{r(\varepsilon)}$, we have, putting $a=\frac{m}{r(\varepsilon)}$ in the inequality \eqref{c1},
\begin{align*}
 &\int_{m/r(\varepsilon)}^\infty t\log\Big(1+\frac{m}{t}\Big)\,\tau(e_n(dt))\\
 \leq&\int_{m/r(\varepsilon)}^\infty t\,\frac{m}{t}\,\tau(e_n(dt))=m\tau\Big(e_n\Big(\Big[\frac{m}{r(\varepsilon)},\infty\Big)\Big)\Big)\\
 =&m^\varepsilon r(\varepsilon)^{1-\varepsilon}\Big(\frac{m}{r(\varepsilon)}\Big)^{1-\varepsilon}\tau\Big(e_n\Big(\Big[\frac{m}{r(\varepsilon)},\infty\Big)\Big)\Big)\\
 \leq&r(\varepsilon)^{1-\varepsilon}\tau(h_n^{1-\varepsilon})m^\varepsilon\leq cr(\varepsilon)^{1-\varepsilon}m^\varepsilon.
\end{align*}
Consequently,
\begin{align*}
 &\tau(h_n\log(m\1+h_n)-h_n\log h_n)=\int_0^\infty t\log\Big(1+\frac{m}{t}\Big)\,\tau(e_n(dt))\\
 =&\int_0^{m/r(\varepsilon)}t\log\Big(1+\frac{m}{t}\Big)\,\tau(e_n(dt))
 +\int_{m/r(\varepsilon)}^\infty t\log\Big(1+\frac{m}{t}\Big)\,\tau(e_n(dt))\\
 \leq&cm^\varepsilon+cr(\varepsilon)^{1-\varepsilon}m^\varepsilon=c\big(1+r(\varepsilon)^{1-\varepsilon}\big)m^\varepsilon,
\end{align*}
which yields
\begin{equation}\label{1}
 \tau(h_n\log h_n-h_n\log(m\1+h_n))\geq-c\big(1+r(\varepsilon)^{1-\varepsilon}\big)m^\varepsilon.
\end{equation}
Next we have, since $\frac{M+t}{M+1}\geq1$ for $t\geq1$,
\begin{align*}
 &\tau(-h_n\log(M+1)+h_n\log(M\1+h_n))=\int_0^\infty t\log\frac{M+t}{M+1}\,\tau(e_n(dt))\\
 =&\int_0^1t\log\frac{M+t}{M+1}\,\tau(e_n(dt))+\int_1^\infty t\log\frac{M+t}{M+1}\,\tau(e_n(dt))\\
 \geq&\int_0^1t\log\frac{M+t}{M+1}\,\tau(e_n(dt))\geq\int_0^1 t\log\frac{M}{M+1}\,\tau(e_n(dt))\\
 =&\log\Big(1-\frac{1}{M+1}\Big)\int_0^1t\,\tau(e_n(dt))\geq\log\Big(1-\frac{1}{M+1}\Big)\tau(h_n).
\end{align*}
For $0\leq u\leq\frac{1}{2}$ we have
\[
 \log(1-u)\geq-(2\log2)u,
\]
thus for $M\geq1$, we get
\[
 \log\Big(1-\frac{1}{M+1}\Big)\geq-\frac{2\log2}{M+1}.
\]
This yields
\begin{equation}\label{2}
 \tau(-h_n\log(M+1)+h_n\log(M\1+h_n))\geq-\frac{2\log2}{M+1}\tau(h_n).
\end{equation}
Putting together the estimates \eqref{1} and \eqref{2} above, we get
\begin{align*}
 &\tau(f(h_n)-f_{m,M}(h_n))=\tau(h_n\log h_n-h_n\log(m\1+h_n))\\
 +&\log(m+1)\tau(h_n)+\tau(-h_n\log(M+1)+h_n\log(M\1+h_n))\\
 \geq&-c(1+r(\varepsilon)^{1-\varepsilon})m^\varepsilon+\log(m+1)\tau(h_n)-\frac{2\log2}{M+1}\tau(h_n).
\end{align*}
Now we have
\begin{align*}
 &\tau(f(h_n))=\tau(f(h_n)-f_{m,M}(h_n))+\tau(f_{m,M}(h_n)-f_{m,M}(h_0))\\
 +&\tau(f_{m,M}(h_0))\geq-c(1+r(\varepsilon)^{1-\varepsilon})m^\varepsilon+\log(m+1)\tau(h_n)+\\
 -&\frac{2\log2}{M+1}\tau(h_n)+\tau(f_{m,M}(h_n)-f_{m,M}(h_0))+\tau(f_{m,M}(h_0)),
\end{align*}
and since
\[
 \tau(f_{m,M}(h_n)-f_{m,M}(h_0))\to0 \quad \text{as }n\to\infty,
\]
we obtain, passing to the limit with $n\to\infty$ in the inequality above,
\begin{equation}\label{linf}
 \begin{aligned}
  \liminf_{n\to\infty}\tau(f(h_n))&\geq-c(1+r(\varepsilon)^{1-\varepsilon})m^\varepsilon+\log(m+1)\tau(h_0)+\\
  &-\frac{2\log2}{M+1}\tau(h_0)+\tau(f_{m,M}(h_0)).
 \end{aligned}
\end{equation}
Now taking $m=\frac{1}{M}$ and again passing to the limit with $M\to\infty$ in the inequality above, we get on account of Lemma \ref{L}
\begin{align*}
 \liminf_{n\to\infty}H(h_n)&=\liminf_{n\to\infty}\tau(f(h_n))\\
 &\geq\lim_{M\to\infty}\tau(f_{m,M}(h_0))=\tau(f(h_0))=H(h_0),
\end{align*}
proving the lower-semicontinuity of $H$.

(3) For every $h\in L^1(\M,\tau)^+$ the entropy of $h$ exists. Now the only difference from the estimates in part (2) is the following estimate
\begin{align*}
 \tau(h_n\log(m\1+h_n)-h_n\log h_n)&=\int_0^\infty t\log\Big(1+\frac{m}{t}\Big)\,\tau(e_n(dt))\\
 &\leq\int_0^\infty t\,\frac{m}{t}\tau(e_n(dt))=m,
\end{align*}
which yields the following inequality analogous to the inequality \eqref{linf}
\begin{equation}\tag{8'}
 \begin{aligned}
  \liminf_{n\to\infty}\tau(f(h_n))&\geq-m+\log(m+1)\tau(h_0)+\\
  &-\frac{2\log2}{M+1}\tau(h_0)+\tau(f_{m,M}(h_0)).
 \end{aligned}
\end{equation}
The rest of the proof is as in point (2).
\end{proof}
From the theorem above, we obtain a number of corollaries of which the first one is a well-known property of von Neumann's entropy, the second provides a sufficient condition for its continuity, and the third is a generalisation of Theorem 10 in \cite{LP1}.
\begin{corollary}\label{C}
Von Neumann's entropy is lower-semicontinuous on the set of the normalised density matrices.
\end{corollary}
Indeed, for a normalised density matrix $h$ the inequality
\[
 \tr h^{1+\varepsilon}\leq1
\]
holds for arbitrary $\varepsilon>0$, thus the corollary follows from point (1) of the theorem taking into account that von Neumann's entropy equals minus Segal's entropy.
\begin{corollary}
Let $\varepsilon>0$ and $c>0$ be arbitrary. Von Neumann's entropy is finite and continuous on the set of normalised density matrices $h$ such that $\tr h^{1-\varepsilon}\leq c$.
\end{corollary}
Indeed, let
\[
 h=\sum_{n=1}^\infty\lambda_ne_n
\]
be the spectral representation of the density matrix $h$. We may assume that $\lambda_1>\lambda_2>\dots$. Since $\lambda_n\to0$, there is $n_0$ such that for $n\geq n_0$ we have
\[
 \log\frac{1}{\lambda_n}\leq\bigg(\frac{1}{\lambda_n}\bigg)^\varepsilon.
\]
Then
\begin{align*}
 -\sum_{n=n_0}^\infty\lambda_n\log\lambda_n\tr e_n&=\sum_{n=n_0}^\infty\lambda_n\log\frac{1}{\lambda_n}\tr e_n \leq\sum_{n=n_0}^\infty\lambda_n\bigg(\frac{1}{\lambda_n}\bigg)^\varepsilon\tr e_n\\
 &=\sum_{n=n_0}^\infty\lambda_n^{1-\varepsilon}\tr e_n\leq\tr h^{1-\varepsilon}\leq c,
\end{align*}
which shows that the von Neumann entropy of $h$ which equals\\ $\displaystyle{-\sum_{n=1}^\infty\lambda_n\log\lambda_n\tr e_n}$ is finite. The upper-semicontinuity of von Neuman's entropy follows from point (2) of the theorem while its lower-semicontinuity follows from Corollary \ref{C}.
\begin{remark}
It is interesting to compare the condition on the continuity of the von Neumann entropy in the corollary above with the same problem considered in \cite{Sh}.
\end{remark}
\begin{corollary}
Let $\M$ be finite, and let $\varepsilon>0$ and $c>0$ be arbitrary. Segal's entropy is finite and continuous on the set
\[
 \{h\in L^1(\M,\tau)^+:\tau(h^{1+\varepsilon})\leq c\}.
\]
\end{corollary}
Indeed, there is $a>1$ such that for $t>a$ we have
\[
 \log t\leq t^\varepsilon.
\]
Then
\begin{align*}
 H(h)&=\int_0^\infty t\log t\,\tau(e(dt))=\int_0^at\log t\,\tau(e(dt))+\int_a^\infty t\log t\,\tau(e(dt))\\
 &\leq a\log a\,\tau(e([0,a]))+\int_a^\infty t^{1+\varepsilon}\,\tau(e(dt))\\
 &\leq a\log a+\tau(h^{1+\varepsilon})\leq a\log a+c<+\infty.
\end{align*}
The continuity of $H$ follows from points (1) and (3) of the theorem.

\section{Ideal-like structure of operators with finite Segal entropy}
It is well-known that the function $t\mapsto\log t$ is operator monotone on $\BH^+$ for finite dimensional $\Ha$. If the dimension of $\Ha$ is infinite, then $\log h$ for selfadjoint positive $h$ is in general unbounded, even for bounded $h$, and obvious problems with domains arise. We shall need the operator monotonicity of the logarithmic function in a simplified version.
\begin{lemma}\label{l}
Let $\M$ be finite. For arbitrary $h_1,h_2\in\widetilde{\M}$ such that\\ $\1\leq h_1\leq h_2$ we have
\[
 \log h_1\leq\log h_2.
\]
\end{lemma}
\begin{proof}
The proof is similar to that for the finite dimensional case and is based on the representation
\[
 \log t=\int_0^\infty\Big(\frac{1}{s+1}-\frac{1}{s+t}\Big)\,ds.
\]
Let
\[
 h_1=\int_1^\infty t\,e_1(dt)
\]
and
\[
 h_2=\int_1^\infty t\,e_2(dt)
\]
be the spectral representations of $h_1$ and $h_2$, respectively. The operators $\log h_1$ and $\log h_2$ are measurable, and for arbitrary\\ $\xi\in\mathcal{D}(\log h_1)\cap\mathcal{D}(\log h_2)$ we have, on account of the Fubini theorem and the fact that the function under the integral sign is nonnegative,
\begin{equation}\label{h1}
 \begin{aligned}
  \langle(\log h_1)\xi|\xi\rangle&=\int_1^\infty\Big(\int_0^\infty\Big(\frac{1}{s+1}-\frac{1}{s+t}\Big)ds\Big)\|e_1(dt)\xi\|^2\\
  &=\int_0^\infty\Big(\int_1^\infty\Big(\frac{1}{s+1}-\frac{1}{s+t}\Big)\|e_1(dt)\xi\|^2\Big)ds\\
  &=\int_0^\infty\Big\langle\Big(\frac{1}{s+1}\1-(s\1+h_1)^{-1}\Big)\xi|\xi\Big\rangle\,ds.
 \end{aligned}
\end{equation}
By the same token, we obtain
\begin{equation}\label{h2}
 \begin{aligned}
  \langle(\log h_2)\xi|\xi\rangle&=\int_1^\infty\Big(\int_0^\infty\Big(\frac{1}{s+1}-\frac{1}{s+t}\Big)ds\Big)\|e_2(dt)\xi\|^2\\
  &=\int_0^\infty\Big\langle\Big(\frac{1}{s+1}\1-(s\1+h_2)^{-1}\Big)\xi|\xi\Big\rangle\,ds.
 \end{aligned}
\end{equation}
Since
\[
 s\1+h_1\leq s\1+h_2,
\]
we have
\[
 \big(s\1+h_2\big)^{-1}\leq\big(s\1+h_1\big)^{-1},
\]
and thus, taking into account the relations \eqref{h1} and \eqref{h2}, we get
\begin{align*}
 \langle(\log h_1)\xi|\xi\rangle&=\int_0^\infty\Big\langle\Big(\frac{1}{s+1}\1-(s\1+h_1)^{-1}\Big)\xi|\xi\Big\rangle\,ds\\
 &\leq\int_0^\infty\Big\langle\Big(\frac{1}{s+1}\1-(s\1+h_2)^{-1}\Big)\xi|\xi\Big\rangle\,ds=\langle(\log h_2)\xi|\xi\rangle.
\end{align*}
This yields that $\log h_1\leq\log h_2$ on $\mathcal{D}(\log h_1)\cap\mathcal{D}(\log h_2)$, and the measurability of $\log h_1$ and $\log h_2$ proves the claim.
\end{proof}
\begin{remark}
The lemma above yields the inequality
\begin{equation}\tag{$\ast$}
 \log h_1\leq\log h_2,
\end{equation}
where $h_1$ and $h_2$ are selfadjoint positive \emph{invertible} operators in $\widetilde{\M}$ such that $h_1\leq h_2$. Indeed, for arbitrary $\varepsilon>0$ and $0\leq h\in L^1(\M,\tau)$, we have
\[
 \log(h+\varepsilon\1)=\log\varepsilon\Big(\frac{1}{\varepsilon}h+\1\Big)=(\log\varepsilon)\1+\log\Big(\frac{1}{\varepsilon}h+\1\Big),
\]
which gives the inequality
\[
 \log(h_1+\varepsilon\1)\leq\log(h_2+\varepsilon\1).
\]
Moreover, for $\xi\in\mathcal{D}(\log h)$, we have
\[
 (\log(h+\varepsilon\1))\xi\underset{\varepsilon\to0}{\longrightarrow}(\log h)\xi,
\]
which implies the inequality
\[
 \log h_1\leq\log h_2
\]
on $\mathcal{D}(\log h_1)\cap\mathcal{D}(\log h_2)$, and thus the inequality ($\ast$).
\end{remark}
\begin{lemma}
Let $\M$ be finite, and let $h\in L^1(\M,\tau)^+$. The entropy of $h$ is finite if and only if the entropy of $h+\1$ is finite.
\end{lemma}
\begin{proof}
Let
\[
 h=\int_0^\infty t\,e(dt)
\]
be the spectral representation of $h$. Then
\[
 H(h)=\int_0^\infty t\log t\,\tau(e(dt)),
\]
and
\[
 H(h+\1)=\int_0^\infty(t+1)\log(t+1)\,\tau(e(dt)).
\]
Since the functions $t\mapsto t\log t$ and $t\mapsto(t+1)\log(t+1)$ are bounded on the interval $[0,1]$, and the measure $\tau(e(\cdot))$ is finite, the integrals $\int_0^1t\log t\,\tau(e(dt))$ and $\int_0^1(t+1)\log(t+1)\,\tau(e(dt))$ are finite. On the interval $(1,\infty)$ we have
\begin{equation}\label{i}
 t\log t\leq(t+1)\log(t+1)\leq2t\log2t=2(t\log2+t\log t),
\end{equation}
and since
\[
 \int_1^\infty t\,\tau(e(dt))\leq\int_0^\infty t\,\tau(e(dt))=\tau(h)<\infty,
\]
the inequalities \eqref{i} show that the integral $\int_1^\infty t\log t\,\tau(e(dt))$ is finite if and only if the integral $\int_1^\infty(t+1)\log(t+1)\,\tau(e(dt))$ is finite which ends the proof.
\end{proof}
The next proposition and Theorem \ref{main} describe some properties of Segal's entropy analogous to those of von Neumann's entropy. It is of interest to notice how their proofs differ from the proofs of these properties for the von Neumann entropy given in Theorem \ref{idvN} and Proposition \ref{vNen} below where some simple facts about the eigenvalues of trace-class operators are employed.
\begin{proposition}\label{fin}
Let $h_1,h_2\in L^1(\M,\tau)^+$ be such that $h_1\leq h_2$, and assume that the entropy of $h_2$ is finite. Then the entropy of $h_1$ is also finite.
\end{proposition}
\begin{proof}
From Lemma \ref{l} we have
\[
 \log(h_1+\1)\leq\log(h_2+\1),
\]
and thus
\begin{equation}\label{in}
 \begin{aligned}
  (h_1+\1)\log(h_1+\1)&=(h_1+\1)^{1/2}(\log(h_1+\1))(h_1+\1)^{1/2}\\
  &\leq(h_1+\1)^{1/2}(\log(h_2+\1))(h_1+\1)^{1/2}.
 \end{aligned}
\end{equation}
Further we have
\begin{align*}
 &(h_2+\1)\log(h_2+\1)=\big(\log^{1/2}(h_2+\1)\big)(h_2+\1)\big(\log^{1/2}(h_2+\1)\big)\\
 \geq&\big(\log^{1/2}(h_2+\1)\big)(h_1+\1)\big(\log^{1/2}(h_2+\1)\big)\\
 =&\big(h_1+\1)^{1/2}\log^{1/2}(h_2+\1)\big)^*\big((h_1+\1)^{1/2}\log^{1/2}(h_2+\1)\big),
\end{align*}
and the finiteness of $H(h_2)$, and thus $H(h_2+\1)$, yields the inequality
\begin{align*}
 &\tau\big(\big((h_1+\1)^{1/2}\log^{1/2}(h_2+\1)\big)^*\big((h_1+\1)^{1/2}\log^{1/2}(h_2+\1)\big)\big)\\
 \leq&\tau((h_2+\1)\log(h_2+\1))<\infty,
\end{align*}
which means that $(h_1+\1)^{1/2}\log^{1/2}(h_2+\1)\in L^2(\M,\tau)$. Consequently, we get
\begin{align*}
 &\tau((h_2+\1)\log(h_2+\1))\\
 \geq&\tau\big(\big((h_1+\1)^{1/2}\log^{1/2}(h_2+\1)\big)^*\big((h_1+\1)^{1/2}\log^{1/2}(h_2+\1)\big)\big)\\
 =&\tau\big(\big((h_1+\1)^{1/2}\log^{1/2}(h_2+\1)\big)\big((h_1+\1)^{1/2}\log^{1/2}(h_2+\1)\big)^*\big)\\
 =&\tau\big((h_1+\1)^{1/2}(\log(h_2+\1))(h_1+\1)^{1/2}\big),
\end{align*}
which together with the relation \eqref{in} gives
\begin{align*}
 \tau((h_1+\1)\log(h_1+\1))&\leq\tau\big((h_1+\1)^{1/2}(\log(h_2+\1))(h_1+\1)^{1/2}\big)\\
 &\leq\tau((h_2+\1)\log(h_2+\1)),
\end{align*}
showing that $h_1+\1$, and thus $h_1$, has finite entropy.
\end{proof}
\begin{lemma}\label{L1}
Let $h\in L^1(\M,\tau)^+$ have finite entropy. Then for arbitrary $z\in\M$ the entropy of $z^*hz$ is also finite.
\end{lemma}
\begin{proof}
Assume first that $\|z\|\leq1$, and define a map $\Phi$ on $L^1(\M,\tau)$ by the formula
\[
 \Phi(x)=z^*xz+(\1-zz^*)^{1/2}x(\1-zz^*)^{1/2}, \quad x\in L^1(\M,\tau).
\]
$\Phi$ is linear positive unital and $\Phi|\M$ is normal. For $x\in L^1(\M,\tau)$, we have $z^*x\in L^1(\M,\tau)$, hence
\[
 \tau(z^*xz)=\tau(zz^*x).
\]
The same holds for $(\1-zz^*)^{1/2}$ instead of $z$, so we obtain
\begin{align*}
 \tau(\Phi(x))&=\tau(z^*xz+(\1-zz^*)^{1/2}x(\1-zz^*)^{1/2}\big)\\
 &=\tau(z^*xz)+\tau\big((\1-zz^*)^{1/2}x(\1-zz^*)^{1/2}\big)\\
 &=\tau(zz^*x)+\tau((\1-zz^*)x)=\tau(x),
\end{align*}
which shows that $\tau$ is $\Phi$-invariant. Let $h\in L^1(\M,\tau)^+$ have finite entropy. By virtue of \cite[Theorem 9]{LP2} we have
\[
 H(\Phi(h))\leq H(h),
\]
so $H(\Phi(h))$ is finite. Since
\[
 z^*hz\leq\Phi(h),
\]
Proposition \ref{fin} yields the finiteness of $H(z^*hz)$.

For arbitrary $z\in\M$, we have
\[
 z^*hz=\bigg(\frac{z}{\|z\|}\bigg)^*\|z\|^2h\bigg(\frac{z}{\|z\|}\bigg),
\]
and since
\begin{align*}
 H\big(\|z\|^2h\big)&=\|z\|^2\tau\big(h\log\big(\|z\|^2h\big)\big)\\
 &=\|z\|^2\big(\log\|z\|^2\tau(h)+H(h)\big)<\infty,
\end{align*}
the first part of the proof yields the claim.
\end{proof}
Let $\Lin\E_{\text{fin}}$ be the linear span of $\E_{\text{fin}}$. Then we have the following theorem about the ideal-like structure of $\Lin\E_{\text{fin}}$.
\begin{theorem}\label{main}
Let $x\in\Lin\E_{\textnormal{fin}}$. Then for arbitrary $y\in\M$ we have\\ $yx\in\Lin\E_{\textnormal{fin}}$ and $xy\in\Lin\E_{\textnormal{fin}}$.
\end{theorem}
\begin{proof}
Since each $x\in\Lin\E_{\text{fin}}$ is a finite linear combination of elements from $\E_{\text{fin}}$, it is enough to show that for any $h\in L^1(\M,\tau)^+$ with finite entropy, and arbitrary $y\in\M$, we have $yh\in\Lin\E_{\text{fin}}$ and $hy\in\Lin\E_{\text{fin}}$. The following formula holds
\begin{equation}\label{il}
 yh=\frac{1}{4}\sum_{k=0}^3i^k(y+i^k\1)h(y+i^k\1)^*
\end{equation}
and similarly for $hy$. On account of Lemma \ref{L1}, all elements of the form $(y+i^k\1)h(y+i^k\1)^*$ for $k=0,1,2,3$ which occur on the right hand side of the equality above have finite entropy which finishes the proof.
\end{proof}
An analogous result for von Neumann's entropy, namely, that $\Lin\E_{\text{fin}}$ is an ideal in $\BH$, is mentioned in \cite{W} together with a sketch of proof. For the sake of completeness and comparison, we give a simple proof below which differs from that in \cite{W}.
\begin{theorem}\label{idvN}
Let $\E_{\textnormal{fin}}$ be the set of all density matrices in $\BH$ having finite von Neumann's entropy, and let $\Lin\E_{\textnormal{fin}}$ be the linear span of $\E_{\textnormal{fin}}$. Then $\Lin\E_{\textnormal{fin}}$ is a two-sided ideal in $\BH$.
\end{theorem}
\begin{proof}
By virtue of the formula \eqref{il}, it is enough to show that for a density matrix $h\in\E_{\text{fin}}$ with the von Neumann entropy $S(h)$, and arbitrary $z\in\M$, the von Neumann entropy of $zhz^*$ is finite, moreover, we may assume that $\|h\|_\infty\leq1$ and $\|z\|_\infty\leq1$. Then $zhz^*\leq\1$. Let $1\geq\lambda_1\geq\lambda_2\geq\dots$, and $1\geq\theta_1\geq\theta_2\geq\dots$ be the eigenvalues of $h$ and $zhz^*$, respectively, arranged in decreasing order. Using\\ \cite[Theorem 1.6]{Si} twice, we get the estimate
\[
 \theta_n\leq\|z\|_\infty^2\lambda_n.
\]
Choose $n_0$ such that for $n\geq n_0$ we have
\[
 \|z\|_\infty^2\lambda_n\leq\frac{1}{e}.
\]
Then for such $n$ we have
\begin{align*}
 -\theta_n\log\theta_n&\leq-\|z\|_\infty^2\lambda_n\log\big(\|z\|_\infty^2\lambda_n\big)\\
 &=-\big(\|z\|_\infty^2\log\|z\|_\infty^2\lambda_n+\|z\|_\infty^2\lambda_n\log\lambda_n\big),
\end{align*}
and since the series
\begin{align*}
 &\sum_{n=1}^\infty-\big(\|z\|_\infty^2\log\|z\|_\infty^2\lambda_n+\|z\|_\infty^2\lambda_n\log\lambda_n\big)\\
 =&-\|z\|_\infty^2\log\|z\|_\infty^2\sum_{n=1}^\infty\lambda_n+\|z\|_\infty^2S(h)
\end{align*}
is convergent, we obtain the convergence of the series $\sum\limits_{n=1}^\infty(-\theta_n\log\theta_n)$, showing the claim.
\end{proof}
Let us note the following result, a counterpart of Proposition \ref{fin}, which seems to belong to the folklore of the field.
\begin{proposition}\label{vNen}
Let $h_1,h_2$ be positive trace-class operators in $\BH$ such that $h_1\leq h_2$, and assume that the von Neumann entropy of $h_2$ is finite. Then the von Neumann entropy of $h_1$ is also finite.
\end{proposition}
\begin{proof}
Rescaling, if necessary, we may assume that $h_2\leq\1$. A simple exercise in Hilbert space theory states that there is an $x\in\BH$ such that $\|x\|\leq1$ and
\[
 h_1^{1/2}=xh_2^{1/2}
\]
(cf. e.g. \cite[Exercise E.2.6]{SZ}). Let $1\geq\theta_1\geq\theta_2\geq\dots$, and $1\geq\lambda_1\geq\lambda_2\geq\dots$ be the eigenvalues of $h_1$ and $h_2$, respectively, arranged in decreasing order. Taking into account that $h_2^{1/2}$ is a compact operator and using again \cite[Theorem 1.6]{Si}, we get the estimate
\[
 \theta_n^{1/2}\leq\lambda_n^{1/2},
\]
i.e.
\[
 \theta_n\leq\lambda_n.
\]
Now the reasoning analogous to that in the proof of Theorem \ref{idvN} yields the claim.
\end{proof}

\section{Topological properties of operators having Segal entropy}
We start with a simple observation.
\begin{theorem}\label{T0}
$\E_{\textnormal{fin}}$ is dense in $L^1(\M,\tau)^+$.
\end{theorem}
\begin{proof}
Take an arbitrary $h\in L^1(\M,\tau)^+$, and let
\[
 h=\int_0^\infty t\,e(dt),
\]
be its spectral decomposition. Let $\varepsilon>0$ be given. Since
\[
 \tau(h)=\int_0^\infty t\,\tau(e(dt))<+\infty,
\]
we can find $0<m<M<+\infty$ such that
\[
 \Big\|\int_0^mt\,e(dt)\Big\|_1=\int_0^mt\,\tau(e(dt))<\varepsilon
\]
and
\[
 \Big\|\int_M^\infty t\,e(dt)\Big\|_1=\int_M^\infty t\,\tau(e(dt))<\varepsilon.
\]
We have
\[
 +\infty>\tau(h)\geq\int_m^\infty t\,\tau(e(dt))\geq m\tau(e([m,+\infty))),
\]
which means that $\tau(e([m,+\infty)))<+\infty$. In particular, it follows that the measure $\tau(e(\cdot))$ is finite on the interval $[m,M]$.
Put
\[
 h'=\int_m^Mt\,e(dt).
\]
Then
\[
 H(h')=\int_m^Mt\log t\,\tau(e(dt)),
\]
and since the function $t\mapsto t\log t$ is bounded on the interval $[m,M]$, the entropy of $h'$ is finite. Moreover,
\begin{align*}
 \|h-h'\|_1&=\Big\|\int_0^mt\,e(dt)+\int_M^\infty t\,e(dt)\Big\|_1\\
 &=\int_0^mt\,\tau(e(dt))+\int_M^\infty t\,\tau(e(dt))<2\varepsilon,
\end{align*}
which shows the claim.
\end{proof}

\begin{lemma}\label{L2}
Let $\al_n>0$ be such that $\displaystyle{\sum_{n=1}^\infty\al_n<+\infty}$. There exists a sequence $(\lambda_n)$ of positive numbers such that
\[
 \sum_{n=1}^\infty\al_n\lambda_n<+\infty \qquad \text{and} \qquad \sum_{n=1}^\infty\al_n\lambda_n\log\lambda_n=+\infty.
\]
\end{lemma}
\begin{proof}
Choose a sequence $(k_n)$ such that $\al_{k_n}\leq\frac{1}{2^n}$, and put
\[
 \lambda_r=\begin{cases}
  \frac{1}{\al_{k_n}n^2} & \text{for $r=k_n$}\\
  1 & \text{otherwise}
 \end{cases}.
\]
Then
\[
 \lambda_{k_n}\geq\frac{2^n}{n^2},
\]
thus we get
\[
 \sum_{n=1}^\infty\al_n\lambda_n=\sum_{n=1}^\infty\frac{1}{n^2}+\sum_{r\ne k_n}\al_r<+\infty,
\]
and
\begin{align*}
 \sum_{n=1}^\infty\al_n\lambda_n\log\lambda_n&\geq\sum_{n=1}^\infty\frac{1}{n^2}\log\frac{2^n}{n^2}\\
 &=\sum_{n=1}^\infty\frac{1}{n^2}(n\log2-2\log n)=+\infty. \qedhere
\end{align*}
\end{proof}

\begin{theorem}\label{T1}
Let $\M$ be finite. $\E_\infty$ is not dense in $L^1(\M,\tau)^+$ if and only if $\M$ is *-isomorphic to a von Neumann algebra acting on a finite dimensional Hilbert space (which, in turn, is equivalent to the equality $\E_\infty=\emptyset$).
\end{theorem}
\begin{proof}
Let $h\in L^1(\M,\tau)^+$ have finite entropy. Take an arbitrary $\varepsilon>0$, and assume that there is $h'\in L^1(\M,\tau)^+$ with infinite entropy. It is immediate that for arbitrary $\al>0$ the entropy of $\al h'$ is infinite too, so we may assume that $\|h'\|_1<\varepsilon$. If $h+h'$ had finite entropy, then on account of Proposition \ref{fin} the entropy of $h'\leq h+h'$ would also be finite, consequently, the entropy of $h+h'$ is infinite. Since $\|(h+h')-h\|_1<\varepsilon$, we have found an element with infinite entropy arbitrarily close to $h$. Thus $\E_\infty$ is not dense in $L^1(\M,\tau)^+$ if and only if all elements in $L^1(\M,\tau)^+$ have finite entropy. Assume that this is the case, and let
\[
 \1=\sum_ne_n
\]
be an arbitrary resolution of identity into nonzero projections. Then this resolution must be finite. If not, then we would have
\[
 \1=\sum_{n=1}^\infty e_n, \qquad \sum_{n=1}^\infty\tau(e_n)=1,
\]
and using Lemma \ref{L} for $\al_n=\tau(e_n)$ we could find $\lambda_n>0$ such that
\[
 \sum_{n=1}^\infty \lambda_n\tau(e_n)<+\infty \qquad \text{and} \qquad \sum_{n=1}^\infty\lambda_n\log\lambda_n\tau(e_n)=+\infty.
\]
Now putting
\[
 h=\sum_{n=1}^\infty\lambda_ne_n,
\]
we would have
\[
 \tau(h)=\sum_{n=1}^\infty \lambda_n\tau(e_n)<+\infty,
\]
meaning that $h\in L^1(\M,\tau)^+$, and
\[
 H(h)=\sum_{n=1}^\infty\lambda_n\log\lambda_n\tau(e_n)=+\infty,
\]
a contradiction.

Fix a finite resolution of identity
\[
 \1=\sum_ne_n.
\]
From the fact that this resolution is finite, it follows that each $e_n$ is a finite sum of minimal projections, consequently, we have
\[
 \1=\sum_{n=1}^kp_n,
\]
where the $p_n$'s are minimal projections. For arbitrary central projection $z$ we have
\[
 z=\sum_{n=1}^kzp_n,
\]
and the minimality of $p_n$ yields either $zp_n=0$ or $zp_n=p_n$. Consequently, each central projection is a finite number of some $p_n$'s, in particular, there are only a finite number of central projections. Taking all products of them we get a finite number of minimal central projections $z_1,\dots,z_m$ such that
\[
 \sum_{i=1}^mz_i=\1.
\]
In the central decomposition
\[
 \M=\M_{z_1}\oplus\dots\oplus\M_{z_m},
\]
each $\M_{z_i}$ is a factor of type I${}_{r_i}$ with finite $r_i$ which follows from the minimality of $z_i$ and the finiteness of $\M$. Since each such factor is *-isomorphic to the full algebra $\mathbb{B}(\mathbb{C}^{r_i})$ of operators on the finite dimensional Hilbert space $\mathbb{C}^{r_i}$, we obtain
\[
 \M=\sum_{i=1}^m{}^\oplus\M_{z_i}\simeq\sum_{i=1}^m{}^\oplus\mathbb{B}(\mathbb{C}^{r^i}).
\]
The algebra $\displaystyle{\sum_{i=1}^m{}^\oplus\mathbb{B}(\mathbb{C}^{r^i})}$ is an algebra of operators acting on the finite dimensional Hil\-bert space $\displaystyle{\bigoplus_{i=1}^m\mathbb{C}^{r_i}}$ which shows the first part of the theorem. The reverse part is obvious since for an algebra acting on a finite dimensional Hilbert space all operators $h\log h$ are bounded, consequently, $H(h)=\tau(h\log h)$ is finite.
\end{proof}
\begin{remark}
Observe that along the lines of the proof of this theorem, we obtain the classical result on the denseness of $\E_\infty$ for von Neumann's entropy (cf. \cite{W}). Namely, in Proposition \ref{vNen} it was shown that for arbitrary positive trace-class operators $h_1$ and $h_2$ such that $h_1\leq h_2$ the finiteness of the von Neumann entropy of $h_2$ yields the finiteness of the von Neumann entropy of $h_1$, so the same reasoning as in the beginning of the proof above gives the conclusion.
\end{remark}
We have the following result on the topological structure of $\E_{\text{fin}}$ which again is analogous to the one for von Neumann's entropy (cf. \cite{W}).
\begin{corollary}
Let $\M$ be finite but not *-isomorphic to a von Neumann algebra acting on a finite dimensional Hilbert space. Then $\E_{\textnormal{fin}}$ is a set of the first category in $L^1(\M,\tau)^+$.
\end{corollary}
Indeed, from Theorem \ref{u-lscont} it follows that the entropy function $H$ is lower-semicontinuous, thus for any $n$ the sets
\[
 \{h:H(h)\leq n\}
\]
are closed. On account of Theorem \ref{T1}, we infer that for every \linebreak $h\in L^1(\M,\tau)^+$ there is $h'\in\E_\infty$ arbitrarily close to $h$ (in particular, $h'\in L^1(\M,\tau)^+\smallsetminus\{h:H(h)\leq n\}$\big) which means that \\ $h\in\overline{L^1(\M,\tau)^+\smallsetminus\{h:H(h)\leq n\}}$. Consequently,
\begin{align*}
 L^1(\M,\tau)^+&=\overline{L^1(\M,\tau)^+\smallsetminus\{h:H(h)\leq n\}}\\
 &=\overline{L^1(\M,\tau)^+\smallsetminus\overline{\{h:H(h)\leq n\}}},
\end{align*}
showing that $\{h:H(h)\leq n\}$ is nowhere dense. Since
\[
 \E_{\text{fin}}=\bigcup_{n=1}^\infty\{h:H(h)\leq n\},
\]
it follows that $\E_{\text{fin}}$ is a countable union of nowhere dense sets, i.e. it is a set of the first category.

\begin{theorem}\label{T2}
Let $\M$ be semifinite and not finite. Then $\E_\infty$ is dense in $L^1(\M,\tau)^+$.
\end{theorem}
\begin{proof}
Let $h\in L^1(\M,\tau)^+$ have spectral decomposition
\[
 h=\int_0^\infty t\,e(dt),
\]
and let $\varepsilon>0$ be given. As in the proof of Theorem \ref{T0}, we can find $0<m<M<+\infty$ such that
\[
 \Big\|\int_0^mt\,e(dt)\Big\|_1=\int_0^mt\,\tau(e(dt))<\varepsilon
\]
and
\[
 \Big\|\int_M^\infty t\,e(dt)\Big\|_1=\int_M^\infty t\,\tau(e(dt))<\varepsilon.
\]
Put
\[
 h_1=\int_0^mt\,e(dt), \quad h_2=\int_m^Mt\,e(dt), \quad h_3=\int_M^\infty t\,e(dt),
\]
so
\[
 h=h_1+h_2+h_3.
\]
Since
\[
 e([0,m))+e([m,+\infty))=\1,
\]
and $\tau(e([m,+\infty)))<+\infty$ which was shown in the proof of Theorem~\ref{T0}, we get $\tau(e([0,m)))=+\infty$. On account of the semifiniteness of $\M$, we can find an infinite sequence $(e_n)$ of nonzero pairwise orthogonal projections such that $e_n\leq e([0,m))$ and $\tau(e_n)<+\infty$. Consider the following three possible cases.\\
1. There are $c_1,c_2$ such that $0<c_1\leq\tau(e_n)\leq c_2<+\infty$ for all $n$.

Take a sequence $(\lambda_n)$ such that $0<\lambda_n<1$,
\[
 \sum_{n=1}^\infty\lambda_n<+\infty \qquad \text{and} \qquad \sum_{n=1}^\infty\lambda_n\log\lambda_n=-\infty.
\]
Choose $n_0$ such that $\lambda_n<m$ for $n\geq n_0$, and
\[
 \sum_{n=n_0}^\infty\lambda_n<\frac{\varepsilon}{c_2}.
\]
Putting
\[
 h'_1=\sum_{n=n_0}^\infty\lambda_ne_n,
\]
we get
\[
 \|h'_1\|_1=\tau(h'_1)=\sum_{n=n_0}^\infty\lambda_n\tau(e_n)\leq c_2\sum_{n=n_0}^\infty\lambda_n<\varepsilon,
\]
and
\[
 H(h'_1)=\sum_{n=n_0}^\infty\lambda_n\log\lambda_n\tau(e_n)\leq c_1\sum_{n=n_0}^\infty\lambda_n\log\lambda_n=-\infty.
\]
2. The sequence $(\tau(e_n))$ is not bounded away from zero.

Choose a subsequence $(k_n)$ such that $\tau(e_{k_n})\leq\frac{1}{2^n}$, and put
\[
 \lambda_n=\frac{1}{n^2\tau(e_{k_n})}.
\]
We have
\[
 \lambda_n\geq\frac{2^n}{n^2}.
\]
Choose $n_0\geq4$ such that $\frac{2^n}{n^2}>M$ for $n\geq n_0$, and
\[
 \sum_{n=n_0}^\infty\frac{1}{n^2}<\varepsilon.
\]
Putting
\[
 h'_1=\sum_{n=n_0}^\infty\lambda_ne_{k_n},
\]
we get
\[
 \|h'_1\|_1=\tau(h'_1)=\sum_{n=n_0}^\infty\lambda_n\tau(e_{k_n})=\sum_{n=n_0}^\infty\frac{1}{n^2}<\varepsilon,
\]
and
\begin{align*}
 H(h'_1)&=\sum_{n=n_0}^\infty\lambda_n\log\lambda_n\tau(e_{k_n})\geq\sum_{n=n_0}^\infty\frac{1}{n^2}\log\frac{2^n}{n^2}\\
 &=\sum_{n=n_0}^\infty\Big(\frac{n\log2}{n^2}-\frac{2\log n}{n^2}\Big)=+\infty.
\end{align*}
3. The sequence $(\tau(e_n))$ is unbounded.

Similarly like in case 2, we choose a subsequence $(k_n)$ such that $\tau(e_{k_n})\geq2^n$, and put
\[
 \lambda_n=\frac{1}{n^2\tau(e_{k_n})}.
\]
We have
\[
 \lambda_n\leq\frac{1}{n^22^n}.
\]
Choose $n_0$ such that $\frac{1}{n^22^n}<m$ for $n\geq n_0$, and
\[
 \sum_{n=n_0}^\infty\frac{1}{n^2}<\varepsilon.
\]
Putting
\[
 h'_1=\sum_{n=n_0}^\infty\lambda_ne_{k_n},
\]
we get
\[
 \|h'_1\|_1=\tau(h'_1)=\sum_{n=n_0}^\infty\lambda_n\tau(e_{k_n})=\sum_{n=n_0}^\infty\frac{1}{n^2}<\varepsilon,
\]
and
\begin{align*}
 H(h'_1)&=\sum_{n=n_0}^\infty\lambda_n\log\lambda_n\tau(e_{k_n})\leq\sum_{n=n_0}^\infty\frac{1}{n^2}\log\frac{1}{n^22^n}\\
 &=-\sum_{n=n_0}^\infty\Big(\frac{2\log n}{n^2}+\frac{n\log2}{n^2}\Big)=-\infty.
\end{align*}

Thus in any case we can find an element $h'_1\in L^1(\M,\tau)^+$ with the spectral decomposition
\[
 h'_1=\sum_n\lambda_ne_n
\]
such that $\|h'_1\|_1<\varepsilon$, the spectrum of $h'_1$ lies either in the interval $[0,m)$ (cases 1 and 3) or in the set $(M,+\infty)\cup\{0\}$ (case 2), and the entropy of $h'_1$ is infinite. Put
\[
 h'=h'_1+h_2\in L^1(\M,\tau)^+.
\]
The spectrum of $h_2$ is contained in the set $[m,M]\cup\{0\}$, and the spectrum of $h'_1$ lies outside the interval $[m,M]$, moreover, for the supports of $h'_1$ and $h_2$ we have
\[
 \s(h'_1)\leq e([0,m)) \qquad \text{and} \qquad \s(h_2)\leq e([m,M]),
\]
which means that $h'_1h_2=0$. Consequently, we have the spectral decomposition
\[
 h'=\sum_n\lambda_ne_n+\int_m^Mt\,e(dt),
\]
which yields the equality
\begin{align*}
 H(h')&=\tau(h'\log h')=\sum_n\lambda_n\log\lambda_n\tau(e_n)+\int_m^Mt\log t\,\tau(e(dt))\\
 &=H(h'_1)+\int_m^Mt\log t\,\tau(e(dt)).
\end{align*}
In the proof of Theorem \ref{T0}, it was shown that the integral\\ $\displaystyle{\int_m^Mt\log t\,\tau(e(dt))}$ is finite which means that the entropy of $h'$ is infinite since the entropy of $h'_1$ is such. Finally, we have
\[
 \|h-h'\|_1=\|h_1+h_3-h'_1\|_1\leq\|h_1\|_1+\|h_3\|_1+\|h'_1\|_1<3\varepsilon,
\]
showing that for any $h\in L^1(\M,\tau)^+$ there is an $h'\in L^1(\M,\tau)^+$ arbitrarily close to $h$ and having infinite entropy which yields the claim.
\end{proof}

\end{document}